\documentclass[11pt,letterpaper]{article}

\usepackage[margin=1.15in,top=1.02in,bottom=1.02in]{geometry}
\usepackage{amsmath,amssymb,amsthm,mathtools}
\usepackage{authblk,needspace}
\usepackage{array,booktabs}
\usepackage{enumitem}
\usepackage{microtype}
\usepackage{xcolor}
\usepackage[colorlinks=true,linkcolor=blue!55!black,citecolor=blue!70!black,urlcolor=blue!60!black]{hyperref}

\hypersetup{
  colorlinks=true,
  linkcolor=red,
  citecolor=cyan!60!blue,
  urlcolor=blue
}

\setlist[enumerate]{leftmargin=2.2em,itemsep=2pt,topsep=4pt}
\allowdisplaybreaks

\newtheoremstyle{paperplain}
  {6pt}{6pt}{\itshape}{}{}{.}{0.5em}{\textbf{#1 #2}\thmnote{ \textup{(#3)}}}
\newtheoremstyle{paperdefinition}
  {6pt}{6pt}{\normalfont}{}{}{.}{0.5em}{\textbf{#1 #2}\thmnote{ \textup{(#3)}}}
\theoremstyle{paperplain}
\newtheorem{theorem}{Theorem}[section]
\newtheorem{lemma}[theorem]{Lemma}

\newtheorem{conjecture}[theorem]{Conjecture}
\theoremstyle{paperdefinition}
\newtheorem{definition}[theorem]{Definition}

\DeclareMathOperator{\mad}{mad}
\DeclareMathOperator{\diam}{diam}
\newcommand{\eps}{\varepsilon}

\title{The diameter of recoloring graphs under a maximum average degree bound}
\author{Ruilin Zheng$^1$}
\author{Junying Lu$^{2,}$\footnote{Corresponding author. junyinglu@njust.edu.cn.}}
\affil{\small $^1$School of Mathematics, Nanjing University, Nanjing 210093, P.R. CHINA\\
$^2$School of Mathematics and Statistics, Nanjing University of Science and Technology,\\
Nanjing 210094, P.R. CHINA}
\date{}

\begin{document}
\maketitle
\vspace{-1.2em}

\begin{abstract}
For a graph $G$, we write $\mad(G)$ for its maximum average degree
and $\diam G$ for its diameter. Let $R_k(G)$ be the graph whose vertices are the proper colorings of $G$ with $k$ colors, where two colorings are adjacent when they differ at one vertex. 
Feghali (JCTB, 2021) proved that, for fixed integers $d,k\ge 1$ with $k\ge d+1$ and every $\varepsilon>0$, every $n$-vertex graph $G$ satisfying $\mad(G)\le d-\varepsilon$ has
$\diam R_k(G)=O_{d,k,\varepsilon}(n(\log n)^{d-1})$.
In this article, we prove that
\[
 \diam R_k(G)=O_{d,k,\varepsilon}\!\left(
 n(\log n)^{\left\lfloor (d-1)/(k-d)\right\rfloor}
 \right),
\]
which extends the result proved by Feghali directly. 
The proof uses a partition into independent layers and removes $k-d$ colors at each recursive stage. 
We also improve the bound on the number of layers and determine the best possible linear coefficient in the forest case. 
More precisely, for $0<\eps<2$, every $n$-vertex graph $G$
with $\mad(G)\le2-\eps$ satisfies $\diam R_3(G)\le\rho_M n$,
where $M=\lfloor2/\eps\rfloor$,
$\rho_M=\max_{1\le m\le M}D_m/m$, and $D_m$ is the largest diameter
of $R_3(T)$ over all trees $T$ on $m$ vertices.
Moreover, the coefficient $\rho_M$ is best possible.
\end{abstract}

\medskip
\noindent\textbf{Keywords.} Graph recoloring; reconfiguration graph; maximum average degree; diameter

\section{Introduction}
All graphs in this article are finite and simple. 
For a graph $G$, we write $\diam G$ for its diameter,
with $\diam G=\infty$ when $G$ is disconnected. Let $k$ be a positive integer. 
A proper $k$-coloring of $G$ is a map $\varphi:V(G)\to \{1,2,\cdots,k\}$ such that adjacent vertices receive distinct colors. 
The \emph{$k$-reconfiguration graph} $R_k(G)$ is the graph whose vertices
are the proper $k$-colorings of $G$, with two colorings adjacent
if and only if they differ at exactly one vertex.

A graph is $d$-degenerate if every nonempty subgraph contains a vertex of degree at most $d$. The \emph{maximum average degree} of a graph $G$ is
\[
 \mad(G)=\max_{\emptyset\ne H\subseteq G}\frac{2|E(H)|}{|V(H)|}.
\]
In particular, $\mad(G)<d+1$ implies that $G$ is $d$-degenerate.
Dyer, Flaxman, Frieze, and Vigoda~\cite{DyerFlaxmanFriezeVigoda2006} and, independently, Cereceda, van den Heuvel, and Johnson~\cite{CerecedaVanDenHeuvelJohnson2008} proved that $R_k(G)$ is connected whenever $G$ is $d$-degenerate and $k\ge d+2$. 
Bonsma and Cereceda~\cite{BonsmaCereceda2009} showed that
superpolynomial recoloring distances and computational hardness
can occur without such restrictions. For $d$-degenerate graphs,
Cereceda~\cite{Cereceda2007} proposed the following conjecture.

\begin{conjecture}[Cereceda~\cite{Cereceda2007}]\label{con:Ce}
For every integer $d\ge0$, every $d$-degenerate graph $G$ on $n$ vertices, and every $k\ge d+2$, the graph $R_k(G)$ has diameter $O_d(n^2)$.
\end{conjecture}

The quadratic order is attained for some chordal graphs~\cite{BonamyJohnsonLignosPatelPaulusma2014}. Several general upper bounds are known. Cereceda~\cite{Cereceda2007} proved a quadratic bound when $k\ge2d+1$. Subsequently, Bousquet and Perarnau~\cite{BousquetPerarnau2016} proved a linear bound when $k\ge2d+2$. Bousquet and Heinrich~\cite{BousquetHeinrich2022} established a bound of order $n^{d+1}$ for every $k\ge d+2$,
as well as a quadratic bound when $k\ge 3(d+1)/2$. Their proof contains a list coloring lemma that treats several colors at the same time. This lemma is one of the ingredients used here.

Stronger statements are known for several special graph classes. 
A chordal graph is a graph in which every cycle of length at least four has a chord. 
Bonamy et al.~\cite{BonamyJohnsonLignosPatelPaulusma2014} confirmed Conjecture \ref{con:Ce} for chordal graphs, 
while Bonamy and Bousquet~\cite{BonamyBousquet2018} proved a quadratic bound when the number of colors is at least two more than the treewidth. 
For planar graphs, Eiben and Feghali~\cite{EibenFeghali2020} obtained a subexponential bound with seven colors. Dvo\v{r}\'ak and Feghali~\cite{DvorakFeghali2020,DvorakFeghali2021} proved linear bounds with ten colors, and with seven colors when the graph has no triangle. Bousquet and Heinrich~\cite{BousquetHeinrich2022} obtained a quadratic bound with five colors for bipartite planar graphs. Bartier et al.~\cite{BartierEtAl2023} established connectivity with four colors for planar graphs of girth five and a linear bound with five colors when the girth is at least six. Cranston and Mahmoud~\cite{CranstonMahmoud2024} proved a quadratic bound with five colors for planar graphs containing no cycles of lengths three or five.

The symbol $\log$ denotes the natural logarithm.
Bousquet and Perarnau~\cite{BousquetPerarnau2016} proved that, for fixed integers $d$ and $k$ with $k\ge d+1$ and fixed $\eps>0$, 
the diameter is polynomial when $\mad(G)\le d-\eps$. 
Feghali~\cite{Feghali2019Sparse} later gave a shorter proof. 
He then obtained the following stronger theorem by using a partition into independent layers~\cite{Feghali2021}.

\begin{theorem}[Feghali~\cite{Feghali2021}]\label{thm:feghali}
Let $d,k\ge1$ be integers with $k\ge d+1$, and let $\eps>0$. If $G$ is a graph on $n\ge1$ vertices and $\mad(G)\le d-\eps$, then
\[
 \diam R_k(G)=O_{d,k,\eps}\bigl(n(\log n)^{d-1}\bigr).
\]
\end{theorem}

The exponent in Theorem~\ref{thm:feghali} does not decrease when $k>d+1$. 
Our main theorem records this dependence. 

\begin{theorem}\label{thm:main}
For all integers \(d \geq 1\) and \(k \geq d+1\), there is a constant \(C_{d,k} > 0\) such that, for every \(0 < \varepsilon \leq 1\), every graph \(G\) on \(n \geq 1\) vertices with \(\operatorname{mad}(G) \leq d - \varepsilon\) satisfies 
\begin{equation}\label{eq:main-explicit}
 \diam R_k(G)\le C_{d,k}n
 \left(1+\frac{(d+\eps+1)^2}{8\eps}\log n\right)^{\left\lfloor\frac{d-1}{k-d}\right\rfloor}.
\end{equation}
Consequently, for fixed $d$, $k$, and $\eps$,
\begin{equation}\label{eq:main-asymptotic}
 \diam R_k(G)=O_{d,k,\eps}\!\left(
 n(\log n)^{\left\lfloor(d-1)/(k-d)\right\rfloor}
 \right).
\end{equation}
\end{theorem}

For $k=d+1$, the exponent in \eqref{eq:main-asymptotic} is $d-1$, as in Theorem~\ref{thm:feghali}. 
For $k=d+2$, it is $\lfloor(d-1)/2\rfloor$. 
The bound is $O(n\log n)$ when $k\ge d+\lfloor(d+1)/2\rfloor$, and it is linear when $k\ge2d$. 
The last range also follows from the theorem of Bousquet and Perarnau~\cite{BousquetPerarnau2016}, because $\mad(G)<d$ implies that $G$ is $(d-1)$-degenerate.
The progression is summarized below.

\begin{center}
\begin{tabular}{@{}ll@{}}
\toprule
number of colors & bound supplied by Theorem~\ref{thm:main} \\
\midrule
$k=d+1$ & $O(n(\log n)^{d-1})$ \\
$k=d+2$ & $O(n(\log n)^{\lfloor(d-1)/2\rfloor})$ \\
$k\ge d+\lfloor(d+1)/2\rfloor$ & $O(n\log n)$ \\
$k\ge2d$ & $O(n)$ \\
\bottomrule
\end{tabular}
\end{center}


For certain parameter ranges, stronger bounds are known even
for list coloring; see, for example,
\cite{Cranston2022Sparse,PanWangLiu2026}.
However, Theorem~\ref{thm:main} provides a bound that applies to all admissible values of $d$, $k$, and $\eps$. 
Our proof combines Feghali's layered partition with the list coloring lemma of Bousquet and Heinrich. 
It relies on a structural bound for partitions into independent sets, where each vertex has few neighbors in later sets. This bound is stated and proved at the start of Section~\ref{sec:3}.

If $\mad(G)<2$, then $G$ is a forest. A recent preprint of Asgarli, Krehbiel, MacLean, and Zaimi~\cite{AsgarliEtAl2026} determines the largest diameter of $R_3(T)$ among trees of a fixed order. It yields the following sharp improvement over Theorem~\ref{thm:main}.

\begin{theorem}\label{thm:forest}
Let $0<\eps<2$, and let $M=\lfloor2/\eps\rfloor$. If $G$ is a graph on $n\ge1$ vertices and $\mad(G)\le2-\eps$, then
\begin{equation}\label{eq:forest-bound}
 \diam R_3(G)\le \rho_M n,
\end{equation}
where
\[
 \rho_M=
 \begin{cases}
  1, & M=1,\\[2pt]
  \dfrac32, & 2\le M\le5,\\[5pt]
  \dfrac{11}{6}, & M=6,\\[6pt]
  \dfrac M4+\dfrac{3}{4M}, & M\ge7 \text{ and } M \text{ is odd},\\[8pt]
  \dfrac M4+\dfrac2M, & M\ge8 \text{ and } M \text{ is even}.
 \end{cases}
\]
Moreover, for every fixed $\eps$, equality in \eqref{eq:forest-bound} holds for infinitely many values of $n$.
\end{theorem}

\section{Preliminaries}\label{sec:2}

For a graph $H$, let $N_H(v)$ denote the set of neighbors of $v$, and $d_H(v)=|N_H(v)|$. 
For $U\subseteq V(H)$, the graph induced by $U$ is denoted by $H[U]$, and $H-U=H[V(H)\setminus U]$. 
If $\varphi$ is a coloring and $U$ is a set of vertices, then $\varphi(U)=\{\varphi(u):u\in U\}$; the restriction of $\varphi$ to an induced subgraph $F$ is denoted by $\varphi|_F$. 
The girth of a graph containing a cycle is the length of its shortest cycle.

A list assignment $L$ on a graph $H$ assigns a set $L(v)$ of colors to every $v\in V(H)$. 
A proper coloring $\varphi$ is an $L$-coloring if $\varphi(v)\in L(v)$ for every vertex. 
An $L$-recoloring sequence is a sequence of proper $L$-colorings in which consecutive colorings differ on exactly one vertex.

Let $V_1,\ldots,V_t$ be a partition of $V(G)$. If $v\in V_i$, a neighbor of $v$ in $V_j$ with $j>i$ is called a later neighbor of $v$.

\begin{definition}\label{def:partition}
The partition $V_1,\ldots,V_t$ is an \emph{$s$-degree partition} of $G$ if every $V_i$ is independent and every vertex has at most $s$ later neighbors. For $1\le b\le t$, the induced graph
\[
 F=G[V_1\cup\cdots\cup V_b]
\]
is called a layered prefix. Its forward degree is at most $r$ if every vertex in $V_i$, where $i\le b$, has at most $r$ neighbors in $V_{i+1}\cup\cdots\cup V_t$.
\end{definition}

Definition~\ref{def:partition} implies that an ordering by increasing set index has at most $s$ later neighbors at every vertex. Conversely, the definition agrees with the recursive construction in which $V_i$ is an independent set of vertices of degree at most $s$ after $V_1,\ldots,V_{i-1}$ have been deleted.

We use the following bound, which is established in the proof
of Lemma 8 of \cite{BousquetHeinrich2022}.

\begin{lemma}[Bousquet and Heinrich~\cite{BousquetHeinrich2022}]\label{lem:BH}
Let $a\ge1$, let $u_1,\ldots,u_m$ be an ordering of the vertices of a graph $H$, and let $L$ be a list assignment such that
\[
 |L(u_i)|\ge
 \bigl|N_H(u_i)\cap\{u_{i+1},\ldots,u_m\}\bigr|+a+1
\]
for every $i$. If the union of all lists contains $p\le2a$ colors, then any two $L$-colorings can be transformed into one another by an $L$-recoloring sequence that recolors each vertex at most $p$ times.
\end{lemma}

We will apply Lemma~\ref{lem:BH} while the vertices outside an induced subgraph are fixed. Let $C$ be the full color set, let $F$ be an induced subgraph of $G$, and let $\varphi$ be a proper $C$-coloring of $G$. For $v\in V(F)$, define
\begin{equation}\label{eq:restricted-lists}
 L_{F,\varphi}(v)=C\setminus\varphi\bigl(N_G(v)\setminus V(F)\bigr).
\end{equation}
Every $L_{F,\varphi}$-recoloring sequence in $F$ remains a proper recoloring sequence in $G$ if all vertices outside $F$ keep their colors.

The next lemma is the main step. It removes a prescribed set of $a$ colors from a layered prefix. Its exponent records how many times the forward degree can be decreased by $a$.

\begin{lemma}\label{lem:remove-colors}
Let $r\ge0$, $a\ge1$, and $t\ge1$ be integers, and let $C$ be a set of $r+a+1$ colors. 
Let $V_1,\ldots,V_t$ be a partition of $V(G)$ into independent sets. 
Suppose that $F=G[V_1\cup\cdots\cup V_b]$ is a layered prefix whose forward degree is at most $r$. 
For every set $X\subseteq C$ with $|X|=a$ and every proper $C$-coloring $\varphi$ of $G$, there is a recoloring sequence starting from $\varphi$ that uses only colors in $C$, fixes every vertex of $V(G)\setminus V(F)$, and ends at a coloring in which no vertex of $F$ has a color in $X$.
Moreover, each vertex is recolored at most $P_{r,a}t^{\lfloor r/a\rfloor}$ times, where $P_{r,a}$ depends only on $r$ and $a$.
\end{lemma}

\begin{proof}
Fix a linear order on $C$ and, for each $1\le i\le t$, a linear
order on $V_i$. Every subset of $C$ is equipped with the induced
order. For fixed $a$, we prove the assertion simultaneously for all
integers $r\ge0$ by induction on $p=\lfloor r/a\rfloor$.

First suppose that $p=0$. Then $0\le r<a$. Keep every vertex outside
$F$ fixed and use the lists in \eqref{eq:restricted-lists}. Order the
vertices of $F$ by increasing set index, using the fixed order within
each set. For $v\in V_i\cap V(F)$, let $\ell(v)$ be the number of
later neighbors of $v$ in $F$, and let $e(v)$ be the number of
neighbors of $v$ outside $F$. Since $F$ is a layered prefix whose
forward degree is at most $r$, we have $\ell(v)+e(v)\le r$. Therefore,
\[
 |L_{F,\varphi}(v)|
 \ge |C|-e(v)
 \ge \ell(v)+a+1.
\]
Moreover, $|C|=r+a+1\le2a$.

There is an $L_{F,\varphi}$-coloring of $F$ that avoids $X$. Indeed,
color the sets from $V_b$ down to $V_1$, using the fixed order within
each set. When a vertex $v$ is considered, only its later neighbors
in $F$ and its fixed neighbors outside $F$ impose restrictions. There
are at most $\ell(v)+e(v)\le r$ such neighbors, whereas
$C\setminus X$ has $r+1$ colors. Hence a color in $C\setminus X$ is
available. The resulting coloring is an $L_{F,\varphi}$-coloring and
avoids $X$. Lemma~\ref{lem:BH} transforms $\varphi|_F$ into this
coloring while recoloring each vertex at most $|C|\le2a$ times. By
the observation following \eqref{eq:restricted-lists}, the same
sequence is valid in $G$. Thus the assertion holds when $p=0$, and
we take $P_{r,a}=2a$ in this case.

Now suppose that $p\ge1$, and assume that the assertion holds for
every integer $r'\ge0$ satisfying $\lfloor r'/a\rfloor<p$. Since
$r\ge a$ and $\lfloor(r-a)/a\rfloor=p-1$, the induction hypothesis
applies with $r-a$ in place of $r$.

We describe one stage of the construction. Let $\gamma$ be the
current coloring, with $\gamma=\varphi$ at the first stage. If no
color of $X$ occurs on $F$, there is nothing to prove. Otherwise,
let $h$ be the smallest index such that a vertex of $V_h$ has a
color in $X$, and write
\[
 U=V_1\cup\cdots\cup V_{h-1}
 \quad\text{and}\quad
 W=\{w\in V_h:\gamma(w)\in X\}.
\]
By the choice of $h$, no color of $X$ occurs on $U$.

Each $w\in W$ has at most $r$ later neighbors, whereas
$C\setminus X$ has $r+1$ colors. Let $c(w)$ be the first color of
$C\setminus X$ that is absent from the later neighbors of $w$. For
each $c\in C\setminus X$, define
\[
 W_c=\{w\in W:c(w)=c\}.
\]
The nonempty sets $W_c$ partition $W$, and at most $r+1$ of them are
nonempty. If $U=\emptyset$, recolor each $w\in W$ with $c(w)$.
These recolorings are proper because $V_h$ is independent and $c(w)$
is absent from every later neighbor of $w$. Thus no color of $X$
remains on $V_h$, and this stage is complete. Hence we may assume
that $U\ne\emptyset$.

Treat the nonempty sets $W_c$ in the order induced by the fixed order
on $C$. Before each set is treated, we maintain the following
properties:
\begin{enumerate}[label=\textup{(\roman*)}]
\item no vertex of $U$ or of an already treated set $W_{c'}$ has a
color in $X$;
\item every vertex of an untreated set $W_{c'}$ still has a color in
$X$;
\item no vertex of $V_{h+1}\cup\cdots\cup V_t$ has been recolored
during the present stage.
\end{enumerate}
These properties hold before the first set is treated. Fix a color
$c$ for which $W_c$ is the next set to be treated. Since $r\ge a$,
choose a set $Y\subseteq C\setminus X$ of size $a$ that contains
$c$; for definiteness, let $Y$ consist of $c$ together with the first
$a-1$ colors of $(C\setminus X)\setminus\{c\}$.

We remove all colors of $Y$ from $U$ in two steps. 
First, process the sets $V_{h-1},\ldots,V_1$ in this order, 
using the fixed vertex order within each set.
When a vertex $u\in U$ is considered, if a color of $X$ is absent
from its later neighbors, recolor $u$ with the first such color; 
otherwise leave $u$ unchanged. Every recoloring is proper.
Indeed, the vertices in the same set are independent, 
the vertices in earlier sets have not yet been processed and still have no color in $X$, 
and the selected color is absent from all later neighbors. 
Let $S$ be the set of vertices of $G$ whose current color belongs to $X$ after this step.

If $u\in U\setminus S$, then every color of $X$ occurs on a later
neighbor of $u$. Since the sets are processed in decreasing index
order, the colors of these later neighbors do not change after $u$
is considered. Thus these $a$ later neighbors are distinct and
belong to $S$. Hence $u$ has at most $r-a$ later neighbors in $G-S$.
With respect to the partition
$V_1\setminus S,\ldots,V_t\setminus S$, after empty sets are omitted,
the graph $(G-S)[U\setminus S]$ is a layered prefix of $G-S$ whose
forward degree is at most $r-a$ and whose partition has at most $t$
sets. Every vertex of $G-S$ has a color in $C\setminus X$, and
\[
 |C\setminus X|=r+1=(r-a)+a+1.
\]

Second, if $U\setminus S=\emptyset$, take the empty recoloring sequence.
Otherwise, apply the induction hypothesis in $G-S$ to the prefix
$(G-S)[U\setminus S]$, with color set $C\setminus X$ and prescribed
set $Y$. When the induction hypothesis is used, it recolors each
vertex at most $P_{r-a,a}t^{p-1}$ times because the new partition has
at most $t$ sets. In either case, all vertices outside the prefix
remain fixed. During this sequence, every vertex of $S$ keeps a
color in $X$, whereas every recolored vertex uses a color in
$C\setminus X$. Hence the sequence is also proper in $G$. At its
end, no color of $Y$ occurs on $U\setminus S$, while every vertex of
$U\cap S$ has a color in $X$. Since $X\cap Y=\emptyset$, no color
of $Y$ occurs on $U$.

We now recolor every vertex of $W_c$ with $c$. No vertex of $U$ has
color $c$. Moreover, by property~\textup{(iii)}, the later neighbors
of every vertex of $W_c$ have kept their colors, so $c$ is still
absent from them. Since $V_h$ is independent, these recolorings are
proper.

It remains to remove the colors of $X$ introduced on $U$. Process
$V_{h-1},\ldots,V_1$ again in the same order. When a vertex $u\in U$
is considered, if a color of $Y$ is absent from its later neighbors,
recolor $u$ with the first such color; otherwise leave $u$ unchanged.
The same argument used for the first procedure shows that every
recoloring is proper, since no color of $Y$ occurs on $U$ before this
procedure begins. Let $T$ be the set of vertices of $G$ whose current
color belongs to $Y$ after this procedure. If $u\in U\setminus T$,
then every color of $Y$ occurs on a later neighbor of $u$. Since the
colors of these later neighbors do not change after $u$ is
considered, the corresponding $a$ neighbors are distinct and belong
to $T$. Thus $u$ has at most $r-a$ later neighbors in $G-T$. With
respect to the partition $V_1\setminus T,\ldots,V_t\setminus T$,
after empty sets are omitted, the graph $(G-T)[U\setminus T]$ is a
layered prefix of $G-T$ whose forward degree is at most $r-a$ and
whose partition has at most $t$ sets.

If $U\setminus T=\emptyset$, take the empty recoloring sequence.
Otherwise, apply the induction hypothesis in $G-T$ to the prefix
$(G-T)[U\setminus T]$, with color set $C\setminus Y$ and prescribed
set $X$. This is allowed because $X\cap Y=\emptyset$ and
\[
 |C\setminus Y|=r+1=(r-a)+a+1.
\]
When the induction hypothesis is used, it recolors each vertex at
most $P_{r-a,a}t^{p-1}$ times because the new partition has at most
$t$ sets. During this sequence, every vertex of $T$ keeps a color in
$Y$, whereas every recolored vertex uses a color in $C\setminus Y$.
Hence the sequence is also proper in $G$. At its end, no color of
$X$ occurs on $U\setminus T$. The vertices of $U\cap T$ have colors
in $Y$, and every vertex of $W_c$ has color $c\in Y$ and remains
fixed. Therefore no color of $X$ occurs on $U\cup W_c$.

The treatment of $W_c$ preserves properties~\textup{(i)--(iii)}:
the vertices of $U$ and of $W_c$ now avoid $X$, all previously
treated sets remain fixed and avoid $X$, every untreated set remains
fixed and still uses colors in $X$, and no vertex in a set after
$V_h$ has been recolored. We may therefore treat the next nonempty
set $W_c$. After all such sets have been treated, no color of $X$
occurs on $U\cup W$. The vertices of $V_h\setminus W$ avoided $X$ at
the beginning of the stage and were never recolored, so no color of
$X$ occurs on $U\cup V_h$. Consequently, the smallest index of a set
containing a color of $X$ strictly increases. Repeating the stage at
most $b\le t$ times removes all colors of $X$ from $F$.

It remains to bound the number of recolorings. During the treatment
of one nonempty set $W_c$, a vertex is recolored at most once in each
of the two deterministic procedures, at most
$P_{r-a,a}t^{p-1}$ times in each of the two applications of the
induction hypothesis, and at most once when the vertices of $W_c$
receive the color $c$. At a stage with $U=\emptyset$, each vertex
is recolored at most once, which is smaller than the same bound.
There are at most $r+1$ nonempty sets $W_c$ at one stage and at most
$t$ stages. Hence every vertex is recolored at most
\begin{equation*}
 (r+1)t\bigl(2P_{r-a,a}t^{p-1}+3\bigr)
\end{equation*}
times. 
Define $P_{r,a}=(r+1)(2P_{r-a,a}+3)$ for $p\ge1$. 
Since $t\ge1$ and $p\ge1$, we have $t\le t^p$.
Therefore 
\[
(r+1)t\bigl(2P_{r-a,a}t^{p-1}+3\bigr) \leq (r+1)\bigl(2P_{r-a,a}t^p+3t^p\bigr)=P_{r,a}t^p.
\]
The constant $P_{r,a}$ depends only on $r$ and $a$. This completes
the induction on $p$ and proves the lemma.
\end{proof}

We next improve the bound on the number of layers.
Instead of first coloring the subgraph induced by vertices of
small degree, we estimate the number of edges in this subgraph
and apply the following elementary fact.
Every nonempty graph $J$ has an independent set of size at least
$\frac{|V(J)|^2}{|V(J)|+2|E(J)|}$.
To verify it, choose a uniformly random ordering of $V(J)$ and
select each vertex that precedes all of its neighbors.
The selected vertices form an independent set.
Its expected size is $\sum_v1/(d_J(v)+1)$, which is at least
the claimed bound by the Cauchy--Schwarz inequality.

\begin{lemma}\label{lem:layers}
Let $d\ge2$ be an integer and let $0<\eps\le1$. If $G$ is a graph on $n\ge1$ vertices and $\mad(G)\le d-\eps$, then $G$ has a $(d-1)$-degree partition with at most
\begin{equation}\label{eq:layer-count}
 1+\frac{(d+\eps+1)^2}{8\eps}\log n
\end{equation}
sets.
\end{lemma}

\begin{proof}
Let $H$ be a nonempty induced subgraph of $G$, and let $h=|V(H)|$.
Define $B=\{v\in V(H):d_H(v)\le d-1\}$, $b=|B|$,
and $O=V(H)\setminus B$.
Write $e_H(B,O)$ for the number of edges with one endpoint in $B$ and the other in $O$. 
Every vertex of $O$ has degree at least $d$. 
Since $H[O]$ is a subgraph of $G$, $2|E(H[O])|\le(d-\eps)|O|$.
Consequently,
\begin{equation}\label{eq:cross-edges}
 e_H(B,O)
 =\sum_{v\in O}d_H(v)-2|E(H[O])|
 \ge\eps|O|=\eps(h-b).
\end{equation}

The sum of the degrees of the vertices in $B$ satisfies
\[
 2|E(H[B])|+e_H(B,O)
 \le(d-\eps)h-d(h-b)=db-\eps h.
\]
Combining this inequality with \eqref{eq:cross-edges} gives
\begin{equation}\label{eq:inside-bound-one}
 2|E(H[B])|\le(d+\eps)b-2\eps h.
\end{equation}

By \eqref{eq:inside-bound-one}, we have $b\ge 2\eps h/(d+\eps)>0$. Moreover, $(d+\eps+1)b-2\eps h\ge b>0$. Applying the preceding independent-set bound to $H[B]$, we obtain an independent set $I\subseteq B$ such that
\[
|I|\ge \frac{b^2}{b+2|E(H[B])|}\ge \frac{b^2}{(d+\eps+1)b-2\eps h}\ge \frac{8\eps h}{(d+\eps+1)^2}.
\]
The last inequality is equivalent to $\bigl((d+\eps+1)b-4\eps h\bigr)^2\ge0$.

Thus every nonempty induced subgraph $H$ contains an independent set $I$ of size at least $8\eps|V(H)|/(d+\eps+1)^2$, and every vertex of $I$ has degree at most $d-1$ in $H$. Starting with $G$, repeatedly delete such an independent set from the current induced subgraph until no vertices remain. The deleted sets are independent, and each vertex has at most $d-1$ neighbors in sets deleted later. Hence they form a $(d-1)$-degree partition.

Since $d\ge2$ and $0<\eps\le1$, we have $0<8\eps/(d+\eps+1)^2<1$. After $j$ deletions, at most
\[
n\left(1-\frac{8\eps}{(d+\eps+1)^2}\right)^j
\le n\exp\left(-\frac{8\eps j}{(d+\eps+1)^2}\right)
\]
vertices remain, where the inequality follows from $1-x\le e^{-x}$ for $x\ge0$. Let $t$ be the number of deleted sets. At least one vertex remains after the first $t-1$ deletions, so $1\le n\exp\bigl(-8\eps(t-1)/(d+\eps+1)^2\bigr)$. Taking logarithms gives $8\eps(t-1)/(d+\eps+1)^2\le\log n$, and hence $t\le1+\frac{(d+\eps+1)^2}{8\eps}\log n$, as required.
\end{proof}

\section{Proofs of Theorems \ref{thm:main} and \ref{thm:forest}}\label{sec:3}

We first state the structural result used in the proof of Theorem~\ref{thm:main}.

\begin{theorem}\label{thm:depth}
Let $s\ge0$, $a\ge1$, and $t\ge1$ be integers. If a graph $G$ on $n\ge1$ vertices has an $s$-degree partition with $t$ sets, then
\begin{equation}\label{eq:depth-bound}
 \diam R_{s+a+1}(G)\le \widetilde C_{s,a}nt^{\lfloor s/a\rfloor},
\end{equation}
where $\widetilde C_{s,a}$ depends only on $s$ and $a$.
\end{theorem}

For $a=1$, Theorem \ref{thm:depth} recovers the $O_s(nt^s)$ bound of
Feghali~\cite[Lemma~2]{Feghali2021}.
For general $a$, the proof combines Feghali's layered method
with the simultaneous treatment of several colors by
Bousquet and Heinrich~\cite{BousquetHeinrich2022}
to obtain the stated bound in terms of $t$.

\begin{proof}[\bfseries{Proof of Theorem~\ref{thm:depth}}]
Let $C$ be a set of $s+a+1$ colors, and let $V_1,\ldots,V_t$ be an $s$-degree partition of $G$. 
For fixed $a$, we prove \eqref{eq:depth-bound} by induction on $p=\lfloor s/a\rfloor$, simultaneously for all integers $s\ge0$ with the same value of $p$.

Suppose first that $s<a$. Order the vertices by increasing set index. 
Every vertex has at most $s$ later neighbors, and the full list $C$ has size $s+a+1$. 
Thus the list condition in Lemma~\ref{lem:BH} holds, and $|C|\le2a$. 
The lemma transforms any two proper $C$-colorings into one another while recoloring each vertex at most $|C|$ times. 
Since each recoloring step changes the color of exactly one vertex, the resulting sequence has length at most $n|C|=(s+a+1)n$.
Taking $\widetilde C_{s,a}=s+a+1$, we obtain \eqref{eq:depth-bound}, since $p=0$ and hence $t^p=1$.

Now let $p\ge1$, and assume that the assertion holds for every
integer $s'\ge0$ satisfying $\lfloor s'/a\rfloor<p$. Let $s$ be an
integer such that $\lfloor s/a\rfloor=p$. Then $s\ge a$ and
$\lfloor(s-a)/a\rfloor=p-1$, so the induction hypothesis applies
with $s-a$ in place of $s$.
Let $\varphi$ and $\psi $ be proper $C$-colorings of $G$, and choose $X\subseteq C$ with $|X|=a$. By Lemma~\ref{lem:remove-colors}, applied with $F=G$, transform $\varphi$ and $\psi $ into colorings $\varphi_1$ and $\psi _1$ that use no color of $X$. In each transformation, every vertex is recolored at most $P_{s,a}t^{\lfloor s/a\rfloor}$ times.

Fix an order of the colors in $X$ and an order of the vertices inside each $V_i$. Starting from $\varphi_1$, process $V_t,V_{t-1},\ldots,V_1$. When a vertex $v$ is considered, recolor it with the first color of $X$ that is absent from its later neighbors, if such a color exists; otherwise leave it unchanged. The recoloring is proper because unprocessed vertices in earlier sets still avoid $X$, vertices in one set are independent, and the selected color is absent from all later neighbors. Apply the same rule to $\psi _1$.

The vertices that receive colors in $X$, and the colors assigned to them, are the same in the two procedures. Indeed, consider the vertices in the common processing order. Before a vertex is considered, the colors from $X$ on all of its later neighbors agree in the two procedures by induction on this order. Hence the rule makes the same decision and, when it recolors the vertex, chooses the same color. Let $S$ be this common set of vertices, and let the resulting colorings be $\varphi_2$ and $\psi _2$.

If $v\notin S$, then every color of $X$ occurs on a later neighbor of $v$. These neighbors are distinct and lie in $S$. Hence $v$ has at most $s-a$ later neighbors in $G-S$. After empty sets are omitted, the sets $V_i\setminus S$ form an $(s-a)$-degree partition of $G-S$ with at most $t$ sets. On $G-S$, both $\varphi_2$ and $\psi _2$ use the color set $C\setminus X$, whose size is
\[
 |C\setminus X|=s+1=(s-a)+a+1.
\]
Let $m=|V(G-S)|$, and let $t'$ be the number of nonempty sets among $V_1\setminus S,\ldots,V_t\setminus S$. 
Thus $t'\le t$. If $m=0$, take the empty recoloring sequence from $\varphi_2|_{G-S}$ to $\psi _2|_{G-S}$. 
Suppose that $m\ge1$. Since $\left\lfloor (s-a)/a\right\rfloor=p-1$, the induction hypothesis gives a recoloring sequence from
$\varphi_2|_{G-S}$ to $\psi _2|_{G-S}$ of length at most
\[
\widetilde C_{s-a,a}m(t')^{p-1}
 \le \widetilde C_{s-a,a}m t^{p-1}.
\]
Every vertex of $S$ keeps a color in $X$, whereas the sequence on $G-S$ uses only colors in $C\setminus X$. 
Hence this sequence is also proper in $G$.

Finally, reverse the sequence from $\psi $ to $\psi _2$. 
We now estimate the total length of the resulting sequence from $\varphi$ to $\psi $. 
Each of the two applications of Lemma~\ref{lem:remove-colors} has length at most $nP_{s,a}t^p$, because every vertex is recolored at most $P_{s,a}t^p$ times. 
Each of the two deterministic procedures has length at most $n$, because every vertex is considered once and is recolored at most once. 
The middle sequence has length at most $\widetilde C_{s-a,a}m t^{p-1}$. 
Therefore the complete sequence has length at most $2nP_{s,a}t^p+2n+\widetilde C_{s-a,a}m t^{p-1}.$
Since $m\le n$, $t\ge1$, and $p\ge1$, this is at most
\[
 \begin{aligned}
 2nP_{s,a}t^p+2nt^p+\widetilde C_{s-a,a}nt^p
 &=
 \bigl(2P_{s,a}+2+\widetilde C_{s-a,a}\bigr)nt^p.
 \end{aligned}
\]
Define $\widetilde C_{s,a}=2P_{s,a}+2+\widetilde C_{s-a,a}$.
Since $\varphi$ and $\psi $ were arbitrary proper $C$-colorings of
$G$, this proves \eqref{eq:depth-bound} for the present value of $s$.
Since $s$ was arbitrary subject to $\lfloor s/a\rfloor=p$, the
assertion holds for every such $s$, and this completes the induction
on $p$.
\end{proof}

We now combine Lemma \ref{lem:layers} with Theorem \ref{thm:depth}
to prove Theorem \ref{thm:main}.
\begin{proof}[\bfseries{Proof of Theorem~\ref{thm:main}}]
If $d=1$, then $\mad(G)\le1-\eps<1$. A graph containing an edge has a subgraph on two vertices with average degree one, so $G$ is edgeless. Any two proper $k$-colorings can then be transformed into one another by recoloring each differing vertex once. Thus $\diam R_k(G)\le n$, and \eqref{eq:main-explicit} follows with $C_{1,k}=1$.

Suppose that $d\ge2$. Lemma~\ref{lem:layers} gives a $(d-1)$-degree partition with a number $t$ of sets satisfying \eqref{eq:layer-count}. Apply Theorem~\ref{thm:depth} with $s=d-1$ and $a=k-d$. Then $s+a+1=k$, and
\[
 \diam R_k(G)
 \le \widetilde C_{d-1,k-d}nt^{\lfloor(d-1)/(k-d)\rfloor}.
\]
Substituting \eqref{eq:layer-count} and taking $C_{d,k}=\widetilde C_{d-1,k-d}$ proves \eqref{eq:main-explicit}. The estimate \eqref{eq:main-asymptotic} follows because $d$, $k$, and $\eps$ are fixed.
\end{proof}

We next prove the sharp linear bound in Theorem \ref{thm:forest}.
\begin{proof}[\bfseries{Proof of Theorem~\ref{thm:forest}}]
Since every cycle has average degree two, the assumption $\mad(G)<2$ implies that $G$ is a forest. Let its components be $T_1,\ldots,T_r$, and write $m_i=|V(T_i)|$. Since $0<\eps<2$, we have $M\ge1$. If $m_i\ge2$, then
\[
 2-\frac2{m_i}=\frac{2|E(T_i)|}{|V(T_i)|}\le2-\eps.
\]
Thus $m_i\le2/\eps$ and hence $m_i\le M$. An isolated vertex also satisfies $m_i=1\le M$.

For $m\ge1$, let $D_m$ be the largest value of $\diam R_3(T)$ over all trees $T$ on $m$ vertices. Asgarli, Krehbiel, MacLean, and Zaimi~\cite[Theorem~1.1 and Appendix~A]{AsgarliEtAl2026} proved that
\[
 (D_1,D_2,D_3,D_4,D_5,D_6)=(1,3,4,6,7,11)
\]
and, for $m\ge7$,
\[
 D_m=
 \begin{cases}
  \dfrac{m^2+3}{4}, & m \text{ is odd},\\[5pt]
  \dfrac{m^2}{4}+2, & m \text{ is even}.
 \end{cases}
\]
For two colorings of $G$, every recoloring step acts in one component, so their distance is at least the sum of the corresponding distances in the components. Conversely, shortest sequences in the components can be performed one after another while all other components are fixed. Hence the distance is the sum of the component distances, and taking the maximum gives
\[
 \diam R_3(G)=\sum_{i=1}^{r}\diam R_3(T_i)
 \le\sum_{i=1}^{r}D_{m_i}
 \le\rho_M\sum_{i=1}^{r}m_i
 =\rho_M n,
\]
where $\rho_M=\max_{1\le m\le M}D_m/m$.

It remains to evaluate this maximum. Let $r_m=D_m/m$. For $1\le m\le6$,
\[
 (r_1,r_2,r_3,r_4,r_5,r_6)
 =\left(1,\frac32,\frac43,\frac32,\frac75,\frac{11}{6}\right).
\]
For $m\ge7$,
\[
 r_m=
 \begin{cases}
  \dfrac m4+\dfrac{3}{4m}, & m \text{ is odd},\\[6pt]
  \dfrac m4+\dfrac2m, & m \text{ is even}.
 \end{cases}
\]
The ratios increase within each parity class. Indeed, for odd $m\ge7$ and even $m\ge8$, respectively,
\[
 r_{m+2}-r_m=\frac12-\frac{3}{2m(m+2)}>0
 \quad\text{and}\quad
 r_{m+2}-r_m=\frac12-\frac4{m(m+2)}>0.
\]
If $M\ge7$ is odd, then
\[
 r_M-r_{M-1}=\frac{M^2-6M-3}{4M(M-1)}>0,
\]
whereas, if $M\ge8$ is even, then
\[
 r_M-r_{M-1}=\frac{M^2+4M-8}{4M(M-1)}>0.
\]
Moreover, $r_7=13/7>11/6=r_6$. Hence, for every $M\ge7$, the maximum of $r_m$ over $1\le m\le M$ is attained at $m=M$. Together with the six initial values, this gives the stated formula for $\rho_M$ and proves \eqref{eq:forest-bound}.

Finally, fix $\eps$ and choose $m_*\le M$ such that $D_{m_*}/m_*=\rho_M$. Let $T_*$ be an $m_*$-vertex tree satisfying $\diam R_3(T_*)=D_{m_*}$, and for $q\ge1$ let $G_q$ be the disjoint union of $q$ copies of $T_*$. Let $H$ be a nonempty subgraph of $G_q$, and write $h=|V(H)|$ and $c$ for the number of components of $H$. Every component of $H$ has at most $m_*$ vertices, so $c\ge h/m_*$. Since $H$ is a forest and $m_*\le M\le2/\eps$,
\[
 \frac{2|E(H)|}{h}\le2-\frac{2c}{h}
 \le2-\frac2{m_*}\le2-\eps.
\]
Thus $\mad(G_q)\le2-\eps$, while
\[
 \diam R_3(G_q)=qD_{m_*}=\rho_M q m_*=\rho_M|V(G_q)|.
\]
Therefore equality holds for infinitely many values of $n$, and the coefficient $\rho_M$ is best possible.
\end{proof}

\section{Concluding remarks}

Theorem~\ref{thm:depth} shows that $a$ colors can be handled together and that the degree toward later sets decreases by $a$ at each recursive stage. If the number of sets is $n$, the theorem gives a bound of order
$n^{1+\lfloor s/a\rfloor}$,
which has the same sequence of exponents as the general method of Bousquet and Heinrich. Under the condition $\mad(G)\le d-\eps$, Lemma~\ref{lem:layers} reduces the number of sets to order $\log n$ and produces the exponent in Theorem~\ref{thm:main}.

The proof does not determine whether $\lfloor s/a\rfloor$ is best possible. 
A further problem is to decide whether additional properties of the partition can yield a linear bound for $k<2d$. 
\medskip

\section*{Acknowledgements}
\noindent 
This research is supported by National Key R\&D Program of China under grant number 2024YFA1013900 and by NSFC under grant number 12471327.

\end{document}